\documentclass[11pt]{article}
\usepackage[usenames]{color}
\usepackage{amsmath}
\usepackage{amssymb}
\usepackage{mathrsfs}
\usepackage{amsfonts}
\usepackage{amsthm}
\usepackage{lscape}
\usepackage{graphicx}
\usepackage{hyperref}
\hypersetup{colorlinks=false,pdftitle="Minimax risks for distributed estimation and sensor networks",pdftex}

\newcommand{\be}{\begin{equation}}
\newcommand{\ee}{\end{equation}}
\newcommand{\ben}{\begin{equation*}}
\newcommand{\een}{\end{equation*}}
\newcommand{\ra}{\rightarrow}
\newtheorem{theorem}{Theorem}
\newtheorem*{theoremprime}{Theorem $1'$}
\newtheorem{lemma}{Lemma}
\newtheorem{corollary}{Corollary}
\date{}

\begin{document}
\title {Contribution to the theory of Pitman estimators}
\author{Abram M. Kagan\thanks{Corresponding author},
Tinghui Yu\\
Department of Mathematics, University of Maryland\\ College Park, MD 20742\\
E-mail: amk@math.umd.edu, yuth@math.umd.edu\\
\\
Andrew Barron, Mokshay Madiman\\
Department of Statistics, Yale University\\
New Haven, CT 06511\\
E-mail: andrew.barron@yale.edu, mokshay.madiman@yale.edu}
\maketitle

\begin{abstract}
\noindent
New inequalities are proved for the variance of the Pitman estimators (minimum variance equivariant estimators)
of $\theta$ constructed from samples of fixed size from populations $F(x-\theta)$. The inequalities are closely
related to the classical Stam inequality for the Fisher information, its analog in small samples, and a
powerful variance drop inequality.
The only condition required is finite variance of $F$; 
even the absolute continuity of $F$ is not assumed.
As corollaries of the main inequalities for small samples, one obtains alternate proofs of known properties of the
Fisher information, as well as interesting new observations like the fact that the variance of the Pitman
estimator based on a sample of size $n$ scaled by $n$  monotonically decreases in $n$.
Extensions of the results to the polynomial versions of the Pitman estimators and a multivariate
location parameter are given. Also, the search for characterization of equality conditions for one of
the inequalities leads to a Cauchy-type functional equation for independent random variables,
and an interesting new behavior of its solutions is described.
\end{abstract}


\section{Introduction}
\label{sec:intro}

Our goal is to present some new inequalities for the variance of the Pitman estimators
of a location parameter from different related samples.

Denote by $t_n$ the Pitman estimator (i.e., the minimum variance equivariant estimator) of $\theta$ from a sample
$(x_1,\ldots,x_n)$ of size $n$ from population $F(x-\theta)$. For simplicity, we first focus on the univariate case,
i.e., $x_i\in {\mathbb R}$. If $\int x^{2}dF(x)<\infty$, the Pitman estimator
can be written as
\be
t_n=\bar x-E(\bar x|x_1-\bar x,\ldots,x_n-\bar x)\label {eq:2}
\ee
where $\bar x$ is the sample mean and $E$ denotes the expectation with respect to $F(x)$ (i. e., when $\theta =0$).

For the univariate case, if $F'=f$ exists, $t_n$ can be also written as
\be
t_n=\frac{\int u\prod_{1}^{n}f(x_i-u)du}{\int \prod_{1}^{n}f(x_i-u)du} \label{eq:3}
\ee
showing that $t_n$ is a generalized Bayes estimator corresponding to an improper prior (uniform on the whole $\mathbb R$).
In this paper the representation (\ref{eq:3}) crucial in studying the behavior of $t_n$ in large samples will not be used.

In Section~\ref{sec:convolve}, we obtain a relationship between the variances of the Pitman estimators
based on data obtained by adding (convolving) the initial samples.
As an application of this inequality, one obtains a new proof of a Fisher information inequality
related to the central limit theorem.
Another application, to distributed estimation using sensor networks, is described elsewhere \cite{MBKY09}.

If $t_n^{(1)},\ldots,t_n^{(N)}$ denote the Pitman estimators from samples of size $n$ from
$F_{1}(x-\theta),\ldots,F_{N}(x-\theta)$, and $t_n$  is the Pitman estimator from a sample of size $n$ from
$F(x-\theta)$ where $F=F_1\ast \ldots \ast F_N$, Kagan \cite{Kag02} showed the superadditivity property
\be
{\rm var}(t_n)\geq {\rm var}(t_n^{(1)})+\ldots+{\rm var}(t_n^{(N)}).\label{eq:11}
\ee
In Section~\ref{sec:char}, we obtain this as a corollary of the main inequality in Section~\ref{sec:convolve},
and study an analytic problem arising in connection with identifying its equality conditions.
In particular, a version of the classical Cauchy functional equation for independent random variables is studied;
the behavior of this equation turns out to be more subtle than in the usual settings.

In Section~\ref{sec:combine}, various inequalities relevant to estimation from a combination of samples
are given. For instance, for the Pitman estimator $t_{m+n}$ constructed from observations
$x_1,\ldots,x_m,y_1,\ldots,y_n$ where the first $m$ observations come from $F(x-\theta)$ and the last $n$ from
$G(x-\theta)$,
\be
\frac{1}{{\rm var}(t_{m+n})}\geq \frac{1}{{\rm var}(t_{m})}+\frac{1}{{\rm var}(t_{n})}\label{eq:10} ,
\ee
where $t_m$ and $t_n$ denote the Pitman estimators constructed from $x_1,\ldots,x_m$ and $y_1,\ldots,y_n$
respectively. A generalization of this inequality has an interesting application to a data pricing problem (where datasets are
to be sold, and the value of a dataset comes from the information it yields about an unknown location parameter);
this application is described by the authors elsewhere \cite{MBKY09:itw}.

As an application of the inequalities proved in Section~\ref{sec:combine}, we prove in Section~\ref{sec:monot}
that for any $n\geq 1$, with $t_n$ now denoting the Pitman estimator constructed from $x_1,\ldots,x_n$ for any $n$,
\be
n{\rm var}(t_n)\geq (n+1){\rm var}(t_{n+1})\label{eq:9}
\ee
with the equality sign holding for an $n\geq 2$ only for a sample from Gaussian population (in which case
$n{\rm var}(t_n)$ is constant in $n$).

If $(x_1,\ldots,x_n)$ is a sample from $s$-variate population $F(x-\theta),\:x,\:\theta\in {\mathbb R}^s$ with $\int_{\mathbb{R}^s}|x|^2dF(x)<\infty$,
the Pitman estimator is defined as the minimum covariance matrix equivariant estimator. Though there is only partial ordering
in the set of covariance matrices, the set of covariance matrices of equivariant estimators has a minimal element which is
the covariance matrix of the Pitman estimator (\ref{eq:2}) of the $s$-variate location parameter. Multivariate extensions
of most of the inequalities mentioned above are given in Section~\ref{sec:multi}.

Assuming $\int x^{2k}dF(x)<\infty $ for some integer $k\geq 1$, the polynomial Pitman
estimator $\hat{t}_{n}^{(k)}$ of degree $k$ is, by definition, the minimum variance equivariant
polynomial estimator (see Kagan \cite{Kag66}). An advantage of the polynomial Pitman
estimator is that it depends only on the first $2k$ moments of $F$. In Section~\ref{sec:poly},
it is shown that the polynomial Pitman estimator preserves almost all the properties of $t_n$
that are studied here.

In Section~\ref{sec:perturb} the setup of observations $x_1,\ldots,x_n$ additively perturbed by independent $y_1,\ldots,y_n$
with self-decomposable distribution function $G(y/\lambda)$ is considered. For  the Pitman estimator $t_{n,\lambda}$ from a sample of size $n$ from
$F_{\lambda}(x-\theta)$ where $F_{\lambda}(x)=\int F(x-u)dG(u/\lambda)$ we prove that ${\rm var}(t_{n,\lambda})$
as a function of $\lambda$, monotonically decreases on $(-\infty,0)$ and increases on $(0,+\infty)$. This
makes rigorous the intuition that adding ``noise'' makes estimation harder.

Section~\ref{sec:discuss} concludes with some discussion of the issues that arise in
considering various possible generalizations of the results presented in this paper.

\subsection{Related literature}

All our results have direct counterparts in terms of the Fisher information, and demonstrate very close
similarities between properties of the inverse Fisher information and the variance of Pitman estimators.

Denote by $I(X)$ the Fisher information on a parameter $\theta\in{\mathbb R}$ contained in an observation $X+\theta$.
Plainly, the information depends only on the distribution $F$ of the noise $X$ but not on $\theta$.

For independent $X,\:Y$ the inequality $I(X+Y)\leq I(X)$ is almost trivial (an observation $X+Y+\theta$ is
``more noisy'' than $X+\theta$). A much less trivial inequality was proved in Stam \cite{Sta59}:
\be
\frac{1}{I(X+Y)}\geq \frac{1}{I(X)}+\frac{1}{I(Y)}. \label{eq:1}
\ee
In Zamir \cite{Zam98}, the Stam inequality is obtained as a direct corollary of the basic properties of the
Fisher information: additivity, monotonicity and reparameterization formula.

The  main inequality in Section~\ref{sec:convolve} is closely related to the classical Stam inequality
for the Fisher information, its version in estimation and a powerful variance
drop inequality proved in a general form in Madiman and Barron \cite{MB07} (described below).
In Artstein {\it et. al.} \cite{ABBN04:1} and Madiman and Barron \cite{MB07} the variance drop inequality led to
improvements of the Stam inequality.

Let now $F(x)=(F_{1}\ast F_{2})(x)=\int F_{1}(y)dF_{2}(x-y)$ and $t'_n,\:t''_n,\;t_n$ be the Pitman estimators from samples of size $n$ from $F_{1}(x-\theta),\:F_{2}(x-\theta)$ and $F(x-\theta)$, respectively. If $\int x^{2}dF(x)<\infty$, the following inequality holds for the variances (Kagan \cite{Kag02}):
\be
{\rm var}(t_n)\geq {\rm var}(t'_n)+{\rm var}(t''_n).\label{eq:4}
\ee
This inequality is, in a sense, a finite sample version of (\ref{eq:1}), as discussed in Kagan \cite{Kag02}.
It is generalized in Section~\ref{sec:convolve}, and its equality conditions are obtained in Section~\ref{sec:char}.

Several of the results in this paper rely on the following {\it variance drop lemma}.

\begin{lemma}\label{lem:vardrop}
Let $X_1,\ldots,X_N$ be independent (not necessarily identically distributed) random vectors.
For ${\bf s}=\{i_1,\ldots,i_m\}\subset \{1,\ldots,N\}$ set ${\bf X_s}=(X_{i_1},\ldots,X_{i_m})$,
with $i_1<i_2<\ldots<i_m$ without loss of generality. For
arbitrary functions $\phi_{\bf s}({\bf X_s})$ with ${\rm var}\{\phi_{\bf s}({\bf X_s})\}<\infty$ and any weights $w_{\bf s}>0,
\sum_{\bf s}w_{\bf s}=1$,
\be
{\rm var}\left\{\sum_{\bf s}w_{\bf s}\phi_{\bf s}({\bf X_s})\right\}\leq {N-1 \choose m-1}\sum_{\bf s}w_{\bf s}^{2}{\rm var}\{
\phi_{\bf s}({\bf X_s})\}\label{eq:5}
\ee
where the summation in both sides is extended over all unordered sets (combinations) $\bf s$ of $m$ elements from
$\{1,\ldots,N\}$.\\
The equality sign in (\ref{eq:5}) holds if and only if all $\phi_{\bf s}({\bf X_s})$ are additively decomposable, i.e.,
\be
\phi_{\bf s}({\bf X_s})=\sum_{i\in{\bf s}}\phi_{{\bf s}i}(X_i).\label{eq:6}
\ee
\end{lemma}

The main idea of the proof goes back to Hoeffding \cite{Hoe48} and is based on an
ANOVA type decomposition, see also Efron and Stein \cite{ES81}. See
Artstein {\it et. al.} \cite{ABBN04:1} for the proof  of Lemma~\ref{lem:vardrop} in case of $m=N-1$, and
Madiman and Barron \cite{MB07} for the general case. In Section~\ref{sec:multi},
we observe that this lemma has a multivariate extension, and use it to prove various
inequalities for Pitman estimation of a multivariate location parameter.

The main inequality of Section~\ref{sec:combine} is also related to Carlen's \cite{Car91} superadditivity
of Fisher information, as touched upon there. See \cite{Kag97} for the statistical meaning and proof of
Carlen's superadditivity.

\section{Convolving independent samples from different populations}
\label{sec:convolve}

Here we first prove a stronger version of superadditivity (\ref{eq:4}).

Let ${\bf x_{k}}=(x_{k1},\ldots,x_{kn}),\:k=1,\ldots,N$ be a sample of size $n$
from population $F_{k}(x-\theta)$. Set
\[\bar x_{k}=\frac{x_{k1}+\ldots+x_{kn}}{n},\quad R_{k}=(x_{k1}-\bar x_k,\ldots,x_{kn}-\bar x_k),\quad\sigma_{k}^2={\rm var}(x_{ki}),\]
and for ${\bf s}=\{i_1,\ldots,i_m\}\subset \{1,\ldots,N \}$,
\[F_{\bf s}(x)=(F_{i_1}\ast \ldots \ast F_{i_m})(x),\quad\bar x_{\bf s}=\sum_{k\in{\bf s}}\bar x_k,
\quad R_{\bf s}=\sum_{k\in {\bf s}}R_k \: ({\rm componentwise}).\]
Also set
\ben\begin{split}
F(x)&=(F_1\ast\ldots\ast F_N)(x),\\
\bar x&=\bar x_1+\ldots+\bar x_N,\\
R&=R_1+\ldots+R_N,\\
\sigma^2&=\sigma_{1}^2+\ldots+\sigma_{N}^2.
\end{split}\een

We will need the following  well known lemma (see, e.g., \cite{Shao03:book, page 41}.
\begin{lemma}\label{lem:cond-exp}
Let $\xi$ be a random variable with $E|\xi|<\infty$ and $\eta_1$, $\eta_2$ arbitrary random elements.
If $(\xi,\eta_1)$ and $\eta_2$ are independent then
\begin{equation}
E(\xi|\eta_1,\eta_2)=E(\xi|\eta_1) \ \ \mathrm{a.s.}
\label{CondIndep}
\end{equation}
\end{lemma}

\begin{theorem}\label{thm:convolve}
Let $t_{{\bf s},n}$ denote the Pitman estimator of $\theta$ from a
sample of size $n$ from $F_{\bf s}(x-\theta)$, and $t_n$ denote the Pitman estimator from a sample of
size $n$ from $F(x-\theta)$.
Under the only condition $\sigma^2<\infty$, for any $n\geq 1$ and any $m$ with $1\leq m\leq N$,
\be
{\rm var}(t_n)\geq \frac{1}{{N-1\choose m-1}}\sum_{\bf s}{\rm var}(t_{{\bf s},n})\label{eq:13}
\ee
where the summation is extended over all combinations $\bf s$ of $m$ elements from $\{1,\ldots,N\}$.
\end{theorem}

\begin{proof}
Set $r={N-1 \choose m-1}$. From the definition (\ref{eq:2}) one has
\[{\rm var}(t_n)=\sigma^2/n -{\rm var}\{E(\bar x|R)\}=\sum_{1}^{N}(\sigma_{k}^2/n)-{\rm var}\bigg\{E\bigg(\sum_{1}^{N}\bar x_k|R\bigg)\bigg\}.\]
Similarly,
\begin{eqnarray*}
(1/r)\sum_{\bf s}{\rm var}(t_{{\bf s},n})&=&(1/r)\sum_{\bf s}\sum_{k\in{\bf s}}\sigma_{k}^2/n-(1/r)\sum_{\bf s}{\rm var}\{E(\bar x_{\bf s}|R_{\bf s})\}\\
&=&\sum_1^N(\sigma_k^2/n)-(1/r)\sum_{\bf s}{\rm var}\{E(\bar x_{\bf s}|R_{\bf s})\},
\end{eqnarray*}
where the last equality is due to the fact that each $k\in\{1,\ldots,N\}$ appears exactly $r$ times in $\bf s$.
On setting $\phi_{\bf s}=E(\bar x_{\bf s}|R_{\bf s})$ and $w_{\bf s}={N \choose m}^{-1}$ for all
$\bf s$ and noticing that so defined $\phi_{\bf s}$ depends only on ${\bf x}_k$, $k\in{\bf s}$,
one has by virtue of Lemma~\ref{lem:vardrop}
\begin{equation}
  r\sum_s\mathrm{var}\{E(\bar x_s|R_s)\} \geq \mathrm{var}\bigg\{\sum_s E(\bar x_s|R_s)\bigg\}
    \label{eq:007}
\end{equation}
Denote by $\bar {\bf s}$ the complement of $\bf s$ in $\{1,\ldots,N\}$. Then $R_{\bf s}$
and $R_{\bar {\bf s}}$ depend on disjoint sets of independent random vectors ${\bf x}_1,\ldots,{\bf x}_N$ and thus
are independent.

By virtue of Lemma~\ref{lem:cond-exp},
\[ \phi_{\bf s}=E(\bar x_{\bf s}|R_{\bf s}, R_{\bar {\bf s}}). \]
From the definition of the $n$-variate vectors $R_{\bf s}$ and $R_{\bar {\bf s}}$ one has
$R=R_{\bf s}+R_{\bar {\bf s}}$. Now due to a well known property of the conditional expectation,
\[E(\bar x_{\bf s}|R)=E[E(\bar x_{\bf s}|R_{\bf s}, R_{\bar {\bf s}})|R]=E[E(\bar x_{\bf s}|R_{\bf s})|R].\]
Since for any random variable $\xi$ and random element $\eta$
\[\mathrm{var}(\xi)\geq\mathrm{var}\{E(\xi|\eta)\},\]
the previous relation results in
\begin{eqnarray}
  \nonumber \mathrm{var}\big\{\sum_{\bf s} E(\bar x_{\bf s}|R_{\bf s})\big\}
  &\geq& \mathrm{var}\{E\big(\sum_{\bf s} E(\bar x_{\bf s}|R_{\bf s})|R\big)\}\\
  \nonumber &=& \mathrm{var}\{E\big(\sum_{\bf s} E(\bar x_{\bf s}|R_{\bf s}, R_{\bar {\bf s}})|R\big)\}\\
  \nonumber &=& \mathrm{var}\{\sum_{\bf s}E( \bar x_{\bf s}|R)\}
  = \mathrm{var}\{E\big(\sum_{\bf s}\bar x_{\bf s}|R\big)\}\\
  &=& \mathrm{var}\{E\big(r\sum_{k=1}^N\bar x_k|R\big)\}
  = r^2\mathrm{var}\{E(\bar x|R)\}. \label{SumVarianceFinal}
\end{eqnarray}
Combining (\ref{eq:007}) with (\ref{SumVarianceFinal}) leads to
\begin{equation}
  \mathrm{var}\{E\big(\sum_{k=1}^{N}\bar x_{k}|R\big)\} \leq \frac{1}{r}\sum_{\bf s}\mathrm{var}\{
    E(\bar x_{\bf s}|R_{\bf s})\}, \label{SumVarianceEquiv}
\end{equation}
which is equivalent to the claimed result (\ref{eq:13}).
\end{proof}

It is of special interest to study the simple case where $F_1=\ldots=F_N=H$. This gives the
monotonicity of $\mathrm{var}(t_n^{\ast N})$ with respect to the
group number $N$, in contrast to (\ref{eq:17}) in Section~\ref{sec:monot}, whose monotonicity is with respect to the
sample size $n$.

\begin{corollary} For any $N>1$, if $t_n^{\ast N}$ is the Pitman estimator
of $\theta$ from a sample of size $n$ from $H^{\ast N}(x-\theta)$ where $H^{\ast N}=H\ast \cdots\ast H$, then
\begin{equation}
\frac{\mathrm{var}(t_n^{\ast N})}{N}\geq \frac{\mathrm{var}(t_n^{\ast (N-1)})}{N-1}.\label{eq:09}
\end{equation}
Here $n$ and $N$ are independent parameters.
\end{corollary}

\begin{proof}
Choose $m=N-1$ in Theorem~\ref{thm:convolve}. Under the conditions of Corollary, $(t_{n,{\bf s}})$ are equidistributed for all $N$ combinations $\bf s$ of $N-1$ elements so that (\ref{eq:13}) becomes
\[\mathrm{var}(t_n^{\ast N})\geq \frac{N}{N-1}\mathrm{var}(t_n^{\ast (N-1)}). \]
\end{proof}

Recall that for independent identically distributed $X_1,\ldots,X_N$, Artstein {\it et. al.} \cite{ABBN04:1}
showed that
\be
NI(X_1+\ldots+X_N)\leq (N-1)I(X_1+\ldots+X_{N-1})\label{eq:8}
\ee
for any $N\geq 1$.
As shown in Ibragimov and Has'minskii \cite{IH81:book}, if  $I(X)<\infty$ and $\int |x|^{\delta} dF(x)<\infty$ for some $\delta>0$,
\begin{equation}
\mathrm{var}(t_n)=\frac{1}{nI(X)}(1+o(1)),\:n\rightarrow\infty \label{MLEVar}
\end{equation}
Thus the inequality (\ref{eq:09}) may be considered a small sample version of
inequality (\ref{eq:8}) for the Fisher information. Furthermore, note that the monotonicity (\ref{eq:8})
of Fisher information follows from (\ref{eq:09}) and (\ref{MLEVar}).


Another corollary of Theorem~\ref{thm:convolve} is a dissipative property of the conditional expectation of the sample mean.

\begin{corollary}\label{WhateverCor}
If $F_1=\ldots=F_N=H$, then for any $N>1$
\begin{equation}
  (N-1)\mathrm{Var}\{E(\bar{x}_1|R_1+\ldots+R_{N-1})\}\geq N\mathrm{Var}\{E(\bar{x}_1|R_1+\ldots+R_N)\}. \label{dissipCond}
  \end{equation}
\end{corollary}

\begin{proof}
Since $x_{11},\ldots,x_{Nn}$ are independent identically distributed random variables, one has for any $n$ and $N$
\begin{eqnarray*}
  \mathrm{var}(t_n^{(N)})&=& \mathrm{var}\{\sum_{k=1}^N\bar{x}_k-E(\sum_{k=1}^N\bar{x}_k|
    R_1+\ldots+R_N)\}\\
  &=& \mathrm{var}\{\sum_{k=1}^N\bar{x}_k\}-\mathrm{var}\{E(\sum_{k=1}^N\bar{x}_k|R_1+\ldots+R_N)\}\\
  &=& \sigma^2/n-\mathrm{var}\{NE(\bar{x}_1|R_1+\ldots+R_N)\}
\end{eqnarray*}
that combined with (\ref{eq:09}) immediately leads to (\ref{dissipCond}).
\end{proof}

Notice that (\ref{dissipCond}) is much stronger than  monotonicity of $\mathrm{var}\{E(\bar{x}_1|R_1+\ldots+R_N)\}$  that follows directly from
\begin{eqnarray*}
  \mathrm{var}\{E(\bar{x}_1|R_1+\ldots+R_{N-1})\}&=&\mathrm{var}\{E(\bar{x}_1|R_1+\ldots+R_{N-1},R_N)\}\\
  &\geq& \mathrm{var}\{E(\bar{x}_1|R_1+\ldots+R_N)\},
\end{eqnarray*}
due to independence of $(\bar{x}_1, R_1,\ldots,R_{N-1})$ and $\mathbf{x}_N$.\\

\section{A corollary and an analytical characterization problem related to the Pitman estimators}
\label{sec:char}

Turn now to an elegant corollary of Theorem~\ref{thm:convolve}. On setting $m=1$ in Theorem~\ref{thm:convolve}, the subsets ${\bf s}$
are reduced to one element each, ${\bf s}=\{k \},\: k=1,\ldots,N$ and   one gets
the superadditivity inequality from Kagan \cite{Kag02}:

\begin{corollary} If $t_n^{(1)},\ldots,t_n^{(N)}$ are the Pitman estimators from samples of size $n$ from
$F_{1}(x-\theta),\ldots,F_{N}(x-\theta)$, and $t_n$  is the Pitman estimator from a sample of size $n$ from
$F(x-\theta)$ where $F=F_1\ast \ldots \ast F_N$, then
\be
{\rm var}(t_n)\geq \sum_{k=1}^N{\rm var}(t_n^{(k)}).\label{AdditivePertThm}
\ee
\end{corollary}

An interesting analytic problem, a Cauchy type functional equation for independent random variables, arises in connection to
the relation
\be
 {\rm var}(t_n)= \sum_{k=1}^N{\rm var}(t_n^{(k)}).\label{eq:A1}
\ee
We will show below that with some conditions on $F_1,\ldots,F_N$,
(\ref{eq:A1}) is a characteristic property of Gaussian distributions.
Note that to study the relation \eqref{eq:A1}, it suffices to consider the case of $N=2$.

Let $(x_1,\ldots,x_n),\:(y_1,\ldots,y_n)$ be independent samples from populations $F_{1}(x-\theta_1),\:F_{2}(y-\theta_2)$,
respectively, and let $t'_{n}$ and $t''_{n}$ be the Pitman estimators of $\theta_1$ and $\theta_2$. The Pitman estimator of $\theta_1 +\theta_2$ from the combined sample $(x_1,\ldots,y_n)$ is $t'_n +t''_n$.

For the Pitman estimator $t_n$ of $\theta$ from a sample of size $n$ from population $(F_1 \ast F_2)(x-\theta)$,
consider $t_{n}(x_1+y_1,\ldots,x_n+y_n)$. This is an equivariant estimator of
$\theta_1+\theta_2$ from the above combined sample, so that
\be
{\rm var}(t_n)\geq {\rm var}(t'_n)+{\rm var}(t''_n).\label{eq:A2}
\ee
Due to the uniqueness of the Pitman estimator, the equality sign in (\ref{eq:A2}) holds if and only if
\be
t_{n}(x_1+y_1,\ldots,x_n+y_n)=t'_{n}(x_1,\ldots,x_n)+t''_{n}(y_1,\ldots,y_n)\label{eq:A3}
\ee
with probability one. This is a Cauchy type functional equation holding for random variables and is different
from the classical Cauchy equation.

It turns out that even in the simplest case of $n=1$ when the equation is of the form
\be
 f(X)+ g(Y)=h(X+Y)\label{eq:A4}
\ee
and $X,\:Y$ are independent continuous random variables, solutions $f,\:g$ of (\ref{eq:A4}) may be nonlinear.

Indeed, let $\xi$ be a uniform random variable on $(0,1)$. Consider its diadic representation
\[ \xi=\sum_{k=1}^\infty\frac{\xi_k}{2^k}, \]
where $\xi_1$, $\xi_2,\ \ldots$ are independent binary random variables with $P(\xi_k=0)=
P(\xi_k=1)=.5$. Now set
\[ X=\sum_{k\ \mathrm{even}}\frac{\xi_k}{2^k},\ Y=\sum_{k\ \mathrm{odd}}\frac{\xi_k}{2^k}. \]
Then $X$ and $Y$ are independent random variables with continuous (though singular)
distributions and they both are functions of $X+Y=\xi$ ($X$ and $Y$ are
strong components of $\xi$, in terminology of Hoffmann-Jorgensen {\it et. al.} \cite{HKPS07}).
Thus, for any measurable functions $f$ and $g$, the relation (\ref{eq:A4})
holds.

On the other hand, if both $X$ and $Y$ have positive almost everywhere (with respect to the Lebesgue measure)
densities and $f$, $g$ are  locally integrable functions, then the equation
(\ref{eq:A4}) has only linear solutions $f$, $g$ (and certainly $h$).

From positivity of the densities, one has
\begin{equation}
f(x)+g(y)=h(x+y)
\label{eq:A5}
\end{equation}
almost everywhere (with respect to the plane Lebesgue measure).  On taking a smooth function $k(x)$ with
compact support, multiplying both sides of (\ref{eq:A5}) by $k(x)$ and integrating
over $x$, results in
\begin{eqnarray*}
  &&\int_{-\infty}^{+\infty}f(x)k(x)dx+g(y)\int_{-\infty}^{+\infty}k(x)dx\\
  &=&\int_{-\infty}^{+\infty}h(x+y)k(x)dx\\
  &=&\int_{-\infty}^{+\infty}h(u)k(u-y)du,
\end{eqnarray*}
where the right hand side is continuous in $y$. Thus, $g(y)$ is continuous and so is
$f(x)$ implying that (\ref{eq:A5}) holds for all (and not almost all) $x,\:y$ (the idea is due to Hillel Furstenberg).

Now (\ref{eq:A5}) becomes the Cauchy classical equation that has only linear solutions.

Returning to (\ref{eq:A3}) and noticing that $E|t'_{n}|<\infty$, $E|t''_{n}|<
\infty$, one concludes that if $F_1$ and $F_2$ are given by almost everywhere
positive densities, then
for almost all (with respect to the Lebesgue measure in ${\mathbb R}^{2n}$)
\be
  t'_{n}(x_1\ldots,x_n)+t''_{n}(y_1,\ldots,y_n)
  =t_n(x_1+y_1,\ldots,x_n+y_n).\label{eq:A6}
\ee
Treating (\ref{eq:A6}) as a Cauchy type equation in $x_i,\:y_i$ with the remaining $n-1$ pairs of the arguments fixed, one gets the linearity of $t'_n,\:t''_n$ in each of their arguments whence due to the symmetry $t'_n=\bar x,\:t''_n=\bar y$
implying for $n\geq 3$ that $F_1$ and $F_2$ are Gaussian. Thus, the following result is proved.

\begin{theorem}\label{thm:char}
Let $t_{n}^{(1)},\ldots,t_{n}^{(N)},\:N>1$ are  the Pitman estimators of $\theta$ from samples of size $n\geq 3$ from populations $F_{1}(x-\theta),\ldots,F_{N}(x-\theta)$ with finite second moments and almost everywhere positive densities, and $t_n$ the Pitman estimator form a sample of size $n$ from $(F_1\ast \ldots \ast F_N)(x-\theta)$. Then
\[{\rm var}(t_n)=\sum_{1}^{N}{\rm var}(t_n^{(k)})\]
if and only if all the populations are Gaussian.
\end{theorem}

\section{Combining independent samples from different populations}
\label{sec:combine}

Let $(x_{1}^{(k)},\ldots,x_{n_k}^{(k)}),\:k=1,\ldots,N$  be independent samples of size $n_1,\ldots,n_N$ from populations                                                                                                                                                                                                                                                                                                                                                                                                                                                                  $F_{1}(x-\theta),\ldots,F_{N}(x-\theta)$ with finite variances and $t_{n_k}^{(k)}$ be the Pitman estimator of $\theta$ from the sample
$(x_{1}^{(k)},\ldots,x_{n_k}^{(k)})$ of size $n_k$. For ${\bf s}=\{i_1,\ldots,i_m\}$, we denote by
$t_{n({\bf s})}^{(\bf s)}$ the Pitman estimator of $\theta$ from the sample of size $n({\bf s})=n_{i_1}+\ldots+n_{i_m}$ that is obtained from pooling samples with superindices from $\bf s$. By $t_{n}^{(1,\ldots,N)}$ we denote the Pitman estimator of $\theta$ from the sample
$(x_{1}^{(1)},\ldots,x_{n_N}^{(N)})$ of size $n=n_1+\ldots+n_N$.
 Trivially, ${\rm var}(t_{n}^{(1,\ldots,N)})$ is the smallest among ${\rm var}(t_{n({\bf s})}^{(\bf s)})$.
 Here a stronger result is proved.

\begin{theorem}\label{thm:combine}
The following inequality holds:
\be
\frac{1}{{\rm var}(t_{n}^{(1,\ldots,N)})}\geq \frac{1}{{N-1 \choose m-1}}\sum_{\bf s}\frac{1}{{\rm var}(t_{n({\bf s})}^{(\bf s)})} \label{eq:4.1}
\ee
where the summation in (\ref{eq:4.1}) is over all combinations $\bf s$ of $m$ elements from $\{1,\ldots,N\}$.
\end{theorem}

\begin{proof}
On setting in Lemma~\ref{lem:vardrop} $\psi_{\bf s}=t_{n({\bf s})}^{(\bf s)}$ and choosing the
weights $w_{\bf s}$ minimizing the right hand side of (\ref{eq:5}),
\[w_{\bf s}=\pi_{\bf s}/\sum_{\bf s}\pi_{\bf s}\]
where $\pi_{\bf s}=1/{\rm var}(t_{n({\bf s})}^{(\bf s)})$,
one gets
\[{N-1 \choose m-1}\frac{1}{\sum_{\bf s}\frac{1}{{\rm var}(t_{n({\bf s})}^{(\bf s)})}}\geq {\rm var}(\sum_{\bf s}w_{\bf s}t_{n({\bf s})}^{(\bf s)}).\]
For sample $(x_{1}^{(1)},\ldots,x_{n_N}^{(N)})$, $\sum _{\bf s}w_{\bf s} t_{n({\bf s})}^{(\bf s)}$ is an equivariant estimator while  $t_{n}^{(1,\ldots,N)}$ is the Pitman estimator. Thus
\[{\rm var}\left(\sum _{\bf s}w_{\bf s} t_{n({\bf s})}^{(\bf s)}\right)\geq {\rm var}(t_{n}^{(1,\ldots,N)})\]
which, combined with the previous inequality, is exactly (\ref{eq:4.1}).
\end{proof}


In a special case, when the subsets $\bf s$ consist of one element and $t_{n_k}^{(k)}$ is
the Pitman estimator from $(x_{1}^{(k)},\ldots,x_{n_k}^{(k)})$, Theorem~\ref{thm:combine} becomes
\be
\frac{1}{{\rm var}(t_{n}^{(1,\ldots,N)})}\geq \frac{1}{{\rm var}(t_{n_1}^{(1)})}+\ldots+\frac{1}
{{\rm var}(t_{n_N}^{(N)})}.\label{eq:4.2}
\ee
This inequality is reminiscent of Carlen's superadditivity for the trace of the Fisher information matrix,
which involves the Fisher informations obtained by taking the limit as sample sizes
go to infinity. However, Carlen's superadditivity is true for random variables with
arbitrary dependence, whereas \eqref{eq:4.2} has only been proved under
assumption of independence of samples.

\section{Some corollaries, including the monotonicity of $n\,{\rm var}(t_n)$}
\label{sec:monot}

Notice that if for a sample of size $m$ from $F(x-\theta)$,
${\rm var}(t_m)<\infty$, then ${\rm var}(t_n)<\infty$ for samples $(x_1,\ldots,x_n)$ of any size $n>m$.

Set $F_1=\ldots=F_N=F$,\, $n_1=\ldots=n_N=1$, and $m=N-1$ in Theorem~\ref{thm:combine}.
Then $n({\bf s})=N-1$ for each ${\bf s}$ with $m$ elements, and $n=N$, and 
Theorem~\ref{thm:combine} reads
\ben
\frac{1}{{\rm var}(t_{N}^{(1,\ldots,N)})}\geq \frac{1}{N-1}\sum_{j}\frac{1}{{\rm var}(t_{N-1}^{(1,\ldots,j-1,j+1,\ldots,N)})}
=\frac{N}{N-1} \frac{1}{{\rm var}(t_{N-1}^{(1,\ldots,N-1)})} ,
\een
where the last equality is due to symmetry.
Now $t_{N}^{(1,\ldots,N)}$ is just the Pitman estimator of $\theta$ from a sample of size $N$
from $F(x-\theta)$. Thus, interpreting $N$ as sample size instead of group size, we have the following result.

\begin{theorem}\label{thm:monot}
Let $t_n$ be the Pitman estimator of $\theta$ from a sample of size $n$ from
a population $F(x-\theta)$. If for some $m$, ${\rm var}(t_m)<\infty$, then for all $n\geq m+1$
\be
(n+1){\rm var}(t_{n+1})\leq n{\rm var}(t_n).\label{eq:17}
\ee
For $n\geq 2$, the equality sign holds if and only if $F$ is Gaussian.
\end{theorem}

{\it Remarks:}
\begin{enumerate}
\item If $F$ is Gaussian $N(0,\sigma^2)$, then clearly $n{\rm var}(t_n)=\sigma ^2$ for all $n$.
In fact, the equality    \[n{\rm var}(t_n)=(n+1){\rm var}(t_{n+1})\]
holding for any $n\geq 2$ characterizes the Gaussian distribution since it implies the
additive decomposability of $t_n$. If an equivariant estimator is additively decomposable,
it is linear and due to the symmetry of $t_n$ one has $t_n=\bar x$.

\item The condition of Theorem~\ref{thm:monot} is fulfilled for $m=1$ (and thus for any $m$) if $\int x^{2}dF(x)<\infty$.
However, for many $F$ with infinite second moment (e.g., Cauchy), ${\rm var}(t_m)<\infty$ for some $m$
and Theorem~\ref{thm:monot} holds.

\item Note that even absolute continuity of $F$ is not required, not to mention the finiteness of the Fisher information.

\item If $F$ is the distribution function of an exponential distribution with parameter $1/\lambda$,
\[n{\rm var}(t_n)=\frac{2\lambda n}{(n+1)(n+2)}.\]
If $F$ is the distribution function of a uniform distribution on $(-1,\:1)$,
\[n{\rm var}(t_n)=\frac{4n^2}{(n+1)^{2}(n+2)}.\]
In these examples, the Fisher information is infinite, but one clearly has monotonicity.

\item One can call $F$ Pitman regular if
\be
\lim_{n\rightarrow \infty}n{\rm var}(t_n)>0 \label{eq:18}
\ee
and nonregular if the limit in (\ref{eq:18}) (that always exists) is zero.
As mentioned earlier, Ibragimov and Has'minskii \cite{IH81:book} showed that under rather mild conditions on $F$
that include the finiteness of the Fisher information $I$,
\[\lim_{n\rightarrow \infty}n{\rm var}(t_n)=1/I.\]
Under these conditions, Theorem~\ref{thm:monot} implies monotone convergence
of $n\,{\rm var}(t_n)$ to its limit.
\end{enumerate}

A corollary of Theorem~\ref{thm:monot} is worth mentioning.

\begin{corollary}\label{cor:wt-combine}
Let  $(x_1,\ldots,x_{n+m})$, $m+n\geq 3$ be a sample from the population
$F(x-\theta)$ with finite variance. If $t_m$ is the Pitman estimator of $\theta$ from the first
$m$ and $t_n$ from the last $n$ observations, then
\[ t_{n+m}=w_1t_n+w_2t_m\] for  some $w_1,\:w_2$
if and only if $F$ is Gaussian.
\end{corollary}

\begin{proof}
One can easily see that necessarily $w_1=m/(m+n),\:w_2=n/(m+n)$ so that
\begin{eqnarray*}
{\rm var}(t_{m+n})=\left(\frac{m}{m+n}\right)^2{\rm var}(t_m)+\left(\frac{n}{m+n}\right)^2{\rm var}(t_n)\\
\geq \left(\frac{m}{m+n}\right){\rm var}(t_{m+n})+\left(\frac{n}{m+n}\right){\rm var}(t_{m+n})={\rm var}(t_{m+n}),
\end{eqnarray*}
the equality sign holding if $(m+n){\rm var}(t_{m+n})=m{\rm var}(t_{m})=n{\rm var}(t_{n})$.
\end{proof}

We can now characterize equality for another special case of Theorem~\ref{thm:combine}.

\begin{corollary}\label{cor:combine-char}
Let $t_m$ be the Pitman estimator from a sample of size $m$ from $F(x-\theta)$.
Then one has superadditivity with respect to the sample size,
\be
\frac{1}{{\rm var}(t_n)}\geq \frac{1}{{\rm var}(t_{n_1})}+\ldots+\frac{1}{{\rm var}(t_{n_N})},\:n=n_1+\ldots+n_N, \label{eq:4.3}
\ee
with equality if and only if $F$ is Gaussian.
\end{corollary}
\begin{proof}
Taking $F_1=\ldots=F_N=F$ in Theorem~\ref{thm:combine} immediately gives \eqref{eq:4.3}.
To understand when the equality sign holds in (\ref{eq:4.3}), suffice to consider the case of $N=2$.
Set $n_1=l,\:n_2=m,\:n=l+m$. The equality sign in
\[\frac{1}{{\rm var}(t_n)}\geq \frac{1}{{\rm var}(t_l)}+\frac{1}{{\rm var}(t_m)}\]
holds if and only if \[t_n=w_1t_l+w_2t_m,\:{\rm with}\:\:w_1=l/n,\:w_2=m/n.\]
According to Corollary~\ref{cor:wt-combine}, the last relation holds if and only if $F$ is Gaussian.
\end{proof}


Another corollary of interest that looks similar in form to Corollary~\ref{WhateverCor} of Section~\ref{sec:convolve} but is of a different nature,
follows immediately from combining Theorem~\ref{thm:monot} and the definition (\ref{eq:2}).

\begin{corollary} For independent identically distributed $X_1,X_2,\ldots$ with\\ ${\rm var}(X_i)=\sigma^2 <\infty$ set
\[\bar X_n=(X_1+\ldots+X_n)/n.\]
Then for any $n\geq 1$,
\[(n+1){\rm var}E(\bar X_{n+1}|X_1-\bar X_{n+1},\ldots, X_{n+1}-\bar X_{n+1})\geq n{\rm var}E(\bar X_{n}|X_1-\bar X_{n},\ldots, X_{n}-\bar X_{n}).\]
\end{corollary}

In the regular case when $\lim_{n\rightarrow \infty} n{\rm var}(t_n)=1/I$,
\[\lim_{n\rightarrow \infty} n{\rm var}E(\bar X_{n}|X_1-\bar X_{n},\ldots, X_{n}-\bar X_{n})=\sigma^2 -1/I.\]
It would be interesting to study the asymptotic behavior as $n\rightarrow\infty$ of the random variable
\[E(\sqrt{n}\bar X_n|X_1-\bar X_{n},\ldots, X_{n}-\bar X_{n}).\]

\section{Multivariate extensions}
\label{sec:multi}


An extension of Theorem~\ref{thm:convolve} to the multivariate case depends on a generalization of the variance drop
lemma (Lemma~\ref{lem:vardrop}) to the case of $s$-variate vector functions. Using the Cram\'er-Wold principle,
for an arbitrary vector $c\in\mathbb{R}^s$ and vector functions $\psi_{\mathbf{s}}=\psi_{\mathbf{s}}(\mathbf{X}_\mathbf{s})$, set
\[ \phi_\mathbf{s}(\mathbf{X}_\mathbf{s})=c^T\psi_\mathbf{s}(\mathbf{X}_\mathbf{s}). \]
Thus Lemma~\ref{lem:vardrop} implies
\[ c^T\mathrm{var}\{\sum_\mathbf{s}\psi_\mathbf{s}\}c\leq {N-1 \choose m-1}\sum_\mathbf{s}w_\mathbf{s}^2c^T\mathrm{var}\{\psi_\mathbf{s}\}c. \]
This is equivalent to
\[ \mathrm{var}\{\sum_\mathbf{s}\psi_\mathbf{s}\}\leq {N-1 \choose m-1}\sum_\mathbf{s}w_\mathbf{s}^2\mathrm{var}\{\psi_\mathbf{s}\}, \]
where var means the covariance matrix; hence Lemma~\ref{lem:vardrop} holds in the multivariate case if we interpret the
inequality in terms of the Loewner ordering.

In Theorem~\ref{thm:convolve}, if $X_1,\ldots,X_n$ are independent $s$-variate random vectors with
distribution $F(x-\theta),\:x,\theta\in\mathbb{R}^s$, all the results and the proof remain true where an
inequality $A\geq B$ for matrices $A,\:B$ means, as usual, that the matrix $A-B$ is non-negative definite.


Corollary~\ref{cor:combine-char} remains valid in the multivariate case when the above samples come from $s$-variate
populations depending on $\theta\in{\mathbb R}^s$ assuming that the covariance matrices of the involved
Pitman estimators are nonsingular. The latter condition is extremely mild. Indeed, if the covariance matrix $V$ of the Pitman estimator
$\tau_n$ from a sample of size $n$ from an $s$-variate population $H(x-\theta)$ is singular, then for a nonzero
(column) vector
$a\in{\mathbb R}^s$
\[{\rm var}(a'\tau_n)=a'Va=0,\]
(prime stands for transposition) meaning that the linear function $a'\theta$ is estimatable with zero variance.
This implies that any two distributions in ${\mathbb R}^{ns}$ generated by samples of size $n$ from $F(x-\theta_1)$
and $F(x-\theta_2)$ with $a'\theta_1\neq a'\theta_2$ are mutually singular and so are the measures in ${\mathbb R}^s$
with distribution functions $F(x-\theta_1)$ and $F(x-\theta_2)$. Since for any $\theta_1$ there exists an arbitrarily close to it $\theta_2$ with $a'\theta_1\neq a'\theta_2$, singularity of the covariance matrix of the Pitman estimator
would imply an extreme irregularity of the family $\{F(x-\theta),\:\theta\in {\mathbb R}^s\}$.
In the multivariate case (\ref{eq:4.2}) takes the form of
\be
V^{-1}(t_{n}^{(1,\ldots,N)})\geq V^{-1}(t_{n_1}^{(1)})+\ldots+
V^{_1}(t_{n_N}^{(N)})\label{eq:4.4}
\ee
where $V(t)$ is the covariance matrix of a random vector $t$. To prove (\ref{eq:4.4}), take matrix-valued weights
\be
W_k=\left(V^{-1}(t_{n_1}^{(1)})+\ldots+V^{-1}(t_{n_N}^{(N)})\right)^{-1}V^{-1}(t_{n_k}^{(k)}),\:k=1,\ldots,N.\label{eq:4.5}
\ee
Since  $W_1+\ldots+W_N$ is the identity matrix, $W_{1}t_{n_1}^{(1)}+\ldots+W_{N}t_{n_N}^{(N)}$ is an equivariant
estimator of $\theta$ so that its covariance matrix exceeds that of the Pitman estimator,
\[V(t_{n}^{(1,\ldots,N)})\leq V \left(W_{1}t_{n_1}^{(1)}+\ldots+W_{N}t_{n_N}^{(N)}\right)=W_1V(t_{n_1}^{(1)})W'_1+\ldots+
W_NV(t_{n_N}^{(N)})W'_N.\]
Substituting the weights (\ref{eq:4.5}) into the last inequality gives (\ref{eq:4.4}).


If $(x_1,\ldots,x_n)$ is a sample from the multivariate population $F(x-\theta)$ (where both $x$ and $\theta$ are vectors),
the monotonicity of Theorem~\ref{thm:monot} holds for the covariance matrix $V_n$ of the Pitman estimator, i.e.,
\[nV_n\geq (n+1)V_{n+1}.\]
The proof is the same as that of the univariate case, but uses the multivariate version of Lemma~\ref{lem:vardrop}
discussed at the beginning of this section.

\section{Extensions to polynomial Pitman estimators}
\label{sec:poly}

Assuming
\be
\int x^{2k}dF(x)<\infty \label{eq:19}
\ee
for some integer $k\geq 1$, the polynomial Pitman
estimator $\hat{t}_{n}^{(k)}$ of degree $k$ is, by definition, the minimum variance equivariant
polynomial estimator (see Kagan \cite{Kag66}). Let $M_k=M_{k}(x_1-\bar x,\ldots,x_n-\bar x)$ be the
space of all polynomials of degree $k$ in the residuals.
Also, let $\hat E(\cdot|M_k)$ be the projection into $M_k$ in the (finite-dimensional) Hilbert space of
polynomials in $x_1,\ldots,x_n$ of degree $k$ with the standard inner product
\[(q_1,q_2)=E(q_1q_2) .\]
Then the polynomial Pitman estimator can be represented as
\be
\hat{t}_{n}^{(k)}=\bar x-\hat E(\bar x|M_k).\label{eq:12}
\ee
Plainly, it depends only on the first $2k$ moments of $F$.


To extend our earlier results to the polynomial Pitman estimators $\hat{t}^{(k)}_n$ under the
assumption $\int x^{2k}dF(x)<\infty$, the following properties of the projection operators are useful:
\begin{enumerate}
\item For any index set $\mathbf{s}$,
\[M_k(R_n)=M_k(R_\mathbf{s}+R_{\bar{\mathbf{s}}})\subset M_k(R_\mathbf{s},R_{\bar{\mathbf{s}}})\]
so that for any random variable $\xi$
\[ {\rm var}\{\hat{E}(\xi|M_k(R_n))\}\leq {\rm var}\{\hat{E}(\xi|M_k(R_\mathbf{s},R_{\bar{\mathbf{s}}}))\}. \]

\item Let $\xi$ be a random variable such that the pair $(\xi,R_\mathbf{s})$ is independent
(actually, suffice to assume uncorrelatedness) of $R_{\bar{\mathbf{s}}}$, then
\[ \hat{E}(\xi|M_k(R_\mathbf{s},R_{\bar{\mathbf{s}}}))=\hat{E}(\xi|M_k(R_\mathbf{s})). \]
\end{enumerate}

Substituting the conditional expectations in the proof of Theorem~\ref{thm:convolve} by the projection operators $\hat{E}(\cdot|M_k)$,
the following  version of Theorem~\ref{thm:convolve} for polynomial Pitman estimators can be proved.

\begin{theoremprime}
If for some integer $k\geq 1$, $\int x^{2k}dF_j(x)<\infty$, $j=1,\ldots,N$, the variance of the
polynomial Pitman estimators $\hat t^{(k)}_{{\bf s},n}$ satisfy the inequality
\[
{\rm var}(\hat t^{(k)}_n)\geq \frac{1}{{N-1\choose m-1}}\sum_{\bf s}{\rm var}(\hat t^{(k)}_{{\bf s},n}).
 \]
\end{theoremprime}


Assuming that for some integer $m\geq 1$
\[\int x^{2m}dF_{k}(x)<\infty,\:k=1,\ldots,N,\]
Corollary~\ref{cor:combine-char} also easily extends 
to the polynomial Pitman estimators of degree $m$.

Similarly, under the condition \eqref{eq:19} for some integer $k\geq 1$, the
Theorem~\ref{thm:monot} extends to the polynomial Pitman estimator $\hat t_{n}^{(k)}$ defined in (\ref{eq:12}).
The polynomial Pitman estimator $\hat t_{n,j}^{(k)}$ of degree $k$ from
$(x_1,\ldots,x_{j-1},x_{j+1},\ldots,x_n)$ is equidistributed with $\hat t_{n-1}^{(k)}$ and thus
\[{\rm var}(\hat t_{n,j}^{(k)})={\rm var}(\hat t_{n-1}^{(k)}).\]
The estimator $(1/n)\sum_{1}^{n}\hat t_{n,j}^{(k)}$ is equivariant  (for sample $(x_1,\ldots,x_n)$) and since $\hat t_n^{(k)}$
is the polynomial Pitman estimator,
\be
{\rm var}(\hat t_{n}^{(k)})\leq (1/n^2){\rm var}(\sum_{1}^{n} \hat t_{n,j}^{(k)}).\label{eq:15}
\ee
By the $m=N-1$ special case of the variance drop lemma, 
\be
{\rm var}(\sum_{1}^{n} \hat t_{n,j}^{(k)})\leq (n-1)\sum_{1}^{n}{\rm var}(\hat t_{n,j}^{(k)})=n(n-1){\rm var}(\hat t_{n-1}^{(k)}).\label{eq:16}
\ee
Combining the last two inequalities gives
\be
(n+1){\rm var}(\hat t_{n+1}^{(k)})\leq n{\rm var}(\hat t_{n}^{(k)}),\label{eq:016}
\ee
i.e., $n{\rm var}(\hat t_{n}^{(k)})$ decreases with $n$.

In Kagan {\it et. al.} \cite{KKF74} it is shown that under only the moment condition (\ref{eq:19}),
$n{\rm var}(\hat t_{n}^{(k)})\ra 1/I^{(k)}$
as $n\ra\infty$ where $I^{(k)}$ can be interpreted as the Fisher information on $\theta$
contained in the first $2k$ moments of $F$ (see Kagan \cite{Kag76}). For any increasing sequence $k(n)$,
one sees that $n{\rm var}(\hat t_n^{k(n)})$ decreases with $n$, and the limit can be equal to $1/I$
under some additional conditions. Indeed, if the span of all the polynomials in $X$ with distribution function $F$
coincides with $L^{2}(F)$, the space of all square integrable functions of $X$,
then $I^{(k)}\rightarrow I$ as $k\rightarrow \infty$.


The above proof of monotonicity is due to the fact that the classes where $t_n$ and $t_{n}^{(k)}$ are the best are rather large. To illustrate this, consider the following analog of $t_{n}^{(k)}$:
\[\tau_{n}^{(k)}=\bar x-\hat E(\bar x|1,m_2,\ldots,m_k)=\bar x -\sum_{j=0}^{k}a_{j,n}m_j\]
where $m_j=(1/n)\sum_{1}^{n}(x_i-\bar x)^j$ and $\hat E(\bar x|1,m_2,\ldots,m_k)$ is the projection of $\bar x$ into the
space span$(1,m_2,\ldots,m_k)$ (i.e., the best mean square approximation of $\bar x$ by linear combinations of
the sample central moments of orders up to $k$). As shown in Kagan {\it et. al.} \cite{KKF74}, if
$\int x^{2k}dF(x)<\infty$, the behavior of $\tau_{n}^{(k)}$ as $n\ra\infty$ is the same as of $t_{n}^{(k)}$:
\[\sqrt{n}(\tau_{n}^{(k)}-\theta)\stackrel{\rm dist}{\ra}Z^{(k)}\]
where $Z^{(k)}$ has a Gaussian distribution $N(0,1/I^{(k)})$ and $n{\rm var}(\tau_{n}^{(k)})\ra 1/I^{(k)}.$
However, it does not seem likely that (\ref{eq:016}) holds for $\tau_{n}^{(k)}.$

\section{Additive perturbations with a scale parameter}
\label{sec:perturb}

In this section the setup of a sample $(x_1,\ldots,x_n)$ from a population $F_{\lambda}(x-\theta)$ is considered where
\[ F_{\lambda}(x)=\int F(y)dG((x-y)/\lambda).\]
In other words, an observation $X$ with distribution function $F(x-\theta)$ is perturbed by an
independent additive noise $\lambda Y$ with $P(Y\leq y)=G(y).$

We study the behavior of the variance ${\rm var}(t_{n,\lambda})$, as a function of $\lambda$,
of the Pitman estimator of $\theta$ from a sample of size $n$
from $F_{\lambda}(x-\theta)$. For the so called self-decomposable $Y$, it is proved that ${\rm var}(t_{n,\lambda})$
behaves ``as expected'', i. e., monotonically decreases for $\lambda\in (-\infty,\:0)$ and increases for $\lambda\in
(0,\:+\infty)$.

They say that a random variable $Y$ is {\it self-decomposable} if for any $c\in (0,1)$, $Y$ is equidistributed with $cY+Z_c$, i.e.,
\be
Y\cong cY+Z_c , \label{eq:5.1}
\ee
where $Z_c$ is independent of $Y$.
If $f(t)$ is the characteristic function of $Y$, then (\ref{eq:5.1}) is equivalent to
\[f(t)=f(ct)g_{c}(t)\]
where $g_{c}(t)$ is a characteristic function. All random variables having stable distributions are self-decomposable.
A self-decomposable random variable is necessarily infinitely divisible. In Lukacs \cite{Luc70:book, Chapter 5} necessary and sufficient conditions are given for self-decomposability in terms of the L\'evy spectral function.

\begin{theorem}\label{thm:decompose}
Let $X$ be an arbitrary random variable with $E(X^2)<\infty$ and $Y$ a self-decomposable random variable
with $E(Y^{2})<\infty$ independent of $X$.
 Then the variance ${\rm var}(t_{n,\lambda})$ of the Pitman estimator
 of $\theta$ from a sample of size $n$ from $F_{\lambda}(x-\theta)$,
 is increasing in $\lambda$ on $(0,\infty)$ and decreasing on $(-\infty,0)$.
\end{theorem}

\begin{proof}
If $x_1,\ldots,x_n,\:y_1,\ldots,y_n$ are independent random variables, the $x$'s  with distribution $F(x-\theta)$ and the $y$'s with distribution $G(y)$, then
\[t_{n,\lambda}=\bar x +\lambda \bar y-E(\bar x +\lambda \bar y|x_1-\bar x+\lambda (y_1-\bar y),\ldots,x_n-\bar x+\lambda (y_n-\bar y))\]
and
\[{\rm var}(t_{n,\lambda})={\rm var}(\bar x +\lambda \bar y)-{\rm var}\{E(\bar x +\lambda \bar y|x_1-\bar x+\lambda (y_1-\bar y),\ldots,x_n-\bar x+\lambda (y_n-\bar y))\}.\]
If $\lambda_2>\lambda_1>0,$ then $\lambda_1=c\lambda_2$ for some $c,\:0<c<1$.

Due to self-decomposability of $y_i$, there exist random variables $z_{c,1}\ldots,z_{c,n}$ such that
\be
y_i-\bar y\cong c(y_i-\bar y)+(z_{c,i}-\bar z_{c})\label{eq:5.2}
\ee
and the random variables $x_1,\ldots,x_n,y_1,\ldots,y_n,z_{c,1},\ldots, z_{c,n}$ are independent.\\
The $\sigma$-algebra
\begin{eqnarray*}
\lefteqn{\sigma(x_1-\bar x+\lambda_{2} (y_1-\bar y),\ldots,x_n-\bar x+\lambda_{2}(y_n-\bar y))=}\\
&& \sigma(x_1-\bar x+\lambda_{2}c (y_1-\bar y)+\lambda_{2}(z_{c,1}-\bar z_c),\ldots,\\
&& x_n-\bar x+\lambda_{2}c(y_n-\bar y)+\lambda_{2}c(y_n-\bar y)+\lambda_{2}(z_{c,n}-\bar z_c))
\end{eqnarray*}
is smaller than the $\sigma$-algebra
\[\sigma(x_1-\bar x+\lambda_{2}c (y_1-\bar y),\ldots,x_n-\bar x+\lambda_{2}c(y_n-\bar y),
z_{c,1}-\bar z_c,\ldots,z_{c,n}-\bar z_c)\]
and thus
\begin{eqnarray*}
\lefteqn{{\rm var}\{E(\bar x +\lambda_{2} \bar y|x_1-\bar x+\lambda_{2} (y_1-\bar y),\ldots,x_n-\bar x+\lambda_{2} (y_n-\bar y))\}\leq}\\
&& {\rm var}\{E(\bar x +\lambda_{2} \bar y|x_1-\bar x+\lambda_{2}c (y_1-\bar y),\ldots,\\
&& x_n-\bar x+\lambda_{2}c(y_n-\bar y),z_{c,1}-\bar z_c,\ldots,z_{c,n}-\bar z_c)\}.
\end{eqnarray*}
From (\ref{eq:5.2}) and Lemma~\ref{lem:cond-exp} in Section~\ref{sec:convolve} one can rewrite the right hand side of the above inequality
\begin{eqnarray}
&&{\rm var}\{E(\bar x +\lambda_{2} \bar y|x_1-\bar x+\lambda_{2}c (y_1-\bar y),\ldots,\nonumber\\
&& x_n-\bar x+\lambda_{2}c(y_n-\bar y),z_{c,1}-\bar z_c,\ldots,z_{c,n}-\bar z_c)\}=\nonumber\\
&& {\rm var}\{E(\bar x +\lambda_{2}c \bar y|x_1-\bar x+\lambda_{2}c (y_1-\bar y),\ldots,x_n-\bar x+\lambda_{2}c(y_n-\bar y)\}+ \nonumber \\
&& {\rm var}\{E(\lambda_{2}\bar z_c|z_{c,1}-\bar z_c,\ldots,z_{c,n}-\bar z_c).\}\label{eq:5.3}
\end{eqnarray}
Again due to (\ref{eq:5.2})
 \[{\rm var}(\bar x+\lambda_{2}\bar y)= {\rm var}(\bar x+\lambda_{2}c\bar y+\lambda_2\bar{z}_c).\]
Combining this with (\ref{eq:5.3}) and recalling that $c\lambda_{2}=\lambda_{1}$ leads to
\[{\rm var}(t_{n,\lambda_{2}})\geq {\rm var}(t_{n,\lambda_{1}}).\]
The case of $\lambda_1<\lambda_2<0$ is treated similarly.
\end{proof}

Theorem~\ref{thm:decompose} has a counterpart in terms of the Fisher information:
{\it Let $X,\:Y$ be independent random variables. If $Y$ is self-decomposable, then
$I(X+\lambda Y)$, as a function of $\lambda$, monotonically increases on $(-\infty,\:0)$ and decreases on $(0,\:+\infty)$.}

The proof is much simpler than that of Theorem~\ref{thm:decompose}. Let $0<\lambda_2=c\lambda_1$ with $0<c<1$. Then $X+\lambda_2 Y\cong
X+c\lambda_2 Y+\lambda_2 Z_c$ where $X,\:Y$ and $Z_c$ are independent and the claim follows from that for independent
random variables $\xi,\:\eta$, $I(\xi+\eta)\leq I(\xi)$. 

\section{Discussion}
\label{sec:discuss}
Few years ago Bulletin of the Institute of Mathematical Statistics published letters \cite{Gup08}, and \cite{Shi08}
whose authors raised a question of monotonicity in the sample size of risks of standard (``classical'') estimators.
Natural expectations are that under reasonable conditions the mean square error, say, of the maximum likelihood estimator
from a sample of size $n+1$ is less than from a sample of size $n$.\\
\\
In this paper a stronger property of the Pitman estimator $t_n$ of a location parameter is proved. Not only
${\rm var}(t_n)$ monotonically decreases in $n$ but ${\rm var}(t_{n+1})\leq\frac{n}{n+1}{\rm var}(t_n)$. However, for another 
equivariant estimator of a location parameter, that is asymptotically equivalent to $t_n$ and has a ``more explicit'' form than $t_n$,
\[\tilde t_n = \bar x -\frac{1}{nI}\sum_{1}^{n}J(x_i -\bar x)\]
where $J$ is the Fisher score and $I$ the Fisher information, monotonicity in $n$ of ${\rm var}(\tilde t_n)$ is an open question.
In a general setup, it is not clear what property of the maximum likelihood estimator is responsible for monotonicity of the risk when monotonicity
holds.\\
\\
In a recent paper \cite{Kag13} was proved monotonicity in the sample size of the length of some confidence intervals.
\\
It seems as a challenge to find out when it is worth to make an extra observation.

\end{document}